\documentclass[11pt]{amsart}

\usepackage{amsmath,amssymb,amsthm,mathtools}
\usepackage[a4paper,margin=1in]{geometry}
\usepackage{hyperref}
\usepackage{enumitem}
\usepackage{mathrsfs}
\usepackage{bbm}
\usepackage{xcolor}

\numberwithin{equation}{section}

\newtheorem{theorem}{Theorem}[section]
\newtheorem{proposition}[theorem]{Proposition}
\newtheorem{lemma}[theorem]{Lemma}

\theoremstyle{definition}

\theoremstyle{remark}
\newtheorem{remark}[theorem]{Remark}

\newcommand{\dd}{\,\mathrm{d}}
\newcommand{\ii}{\mathrm{i}}
\newcommand{\abs}[1]{\left|#1\right|}
\newcommand{\norm}[1]{\left\|#1\right\|}
\newcommand{\E}{\mathbb{E}}
\newcommand{\Var}{\mathrm{Var}}

\newcommand{\Tr}{\mathrm{Tr}}
\newcommand{\op}{\mathrm{op}}
\newcommand{\HS}{\mathrm{HS}}
\newcommand{\Lip}{\mathrm{Lip}}

\newcommand{\Hh}{\mathcal{H}}

\newcommand{\ots}{\otimes_1}

\title[Quantitative CLT for an integrated periodogram]{Quantitative central limit theorem for an integrated periodogram via the fourth moment theorem}

\author{Samir Ben Hariz}
\address{Le Mans Universit\'e, LMM, Le Mans, France}
\email{samir.ben\_hariz@univ-lemans.fr}

\author{Duc-Quang Bui}
\address{Le Mans Universit\'e, LMM, Le Mans, France}
\email{duc\_quang.bui@univ-lemans.fr}

\author{Youssef Esstafa}
\address{Le Mans Universit\'e, LMM, Le Mans, France}
\email{youssef.esstafa@univ-lemans.fr}

\date{}

\begin{document}

\begin{abstract}
We revisit the central limit theorem for integrated periodograms, equivalently for Toeplitz quadratic forms of stationary Gaussian sequences.
Under a regular-variation assumption allowing long-memory singularities and slowly varying corrections, we prove a quantitative central limit theorem in 1-Wasserstein distance.
The proof uses a second Wiener chaos representation and the Malliavin-Stein method (in particular, the Fourth Moment Theorem), reducing normal approximation to
(i) variance asymptotics and (ii) an explicit control of the fourth cumulant via trace estimates for an associated integral operator.
For convenience, we provide self-contained kernel estimates (Dirichlet-type bounds, convolution inequalities, and a weighted Schur test) used in the argument.
\end{abstract}


\maketitle

\section{Introduction}
Let $(X_t)_{t\in\mathbb{Z}}$ be a centered, real-valued, stationary Gaussian sequence with spectral density $f$ on $(-\pi,\pi)$. Define the discrete Fourier transform and the periodogram by
\begin{equation*}
d_n(\lambda)=\frac{1}{\sqrt{2\pi n}}\sum_{t=1}^n X_t e^{\ii t\lambda},
 \qquad
 I_n(\lambda)=|d_n(\lambda)|^2,\qquad \lambda\in(-\pi,\pi).
\end{equation*}
For an even real weight $g$,  consider the centered integrated periodogram
\begin{equation}\label{eq:intperiodogram}
 F_n = \sqrt{n}\int_{-\pi}^{\pi} g(\lambda)\Big(I_n(\lambda)-\E[I_n(\lambda)]\Big)\dd\lambda.
\end{equation}
These quadratic statistics appear naturally in frequency-domain inference (spectral means and Whittle-type criteria). To connect $F_n$ with Toeplitz quadratic forms, observe that since $I_n(\lambda)=(2\pi n)^{-1}\sum_{t,s=1}^n X_tX_s\,e^{\ii(t-s)\lambda}$, we have
\[
\int_{-\pi}^{\pi} g(\lambda)I_n(\lambda)\,\dd\lambda
=\frac{1}{2\pi n}\sum_{t,s=1}^n X_tX_s \int_{-\pi}^{\pi} g(\lambda)e^{\ii(t-s)\lambda}\,\dd\lambda
=\frac{1}{n}\sum_{t,s=1}^n \gamma_g(t-s)\,X_tX_s,
\]
where $\gamma_g(k)=\frac{1}{2\pi}\int_{-\pi}^{\pi} g(\lambda)e^{\ii k\lambda}\dd\lambda$. Thus $F_n$ is a centered Toeplitz-type quadratic form in the sample $X_{1:n}$.

A canonical long-memory setting assumes power singularities at the origin of the form
\[
f(\lambda)=|\lambda|^{-\alpha}L_f(\lambda),
\qquad
g(\lambda)=|\lambda|^{-\beta}L_g(\lambda),
\qquad
d:=\alpha+\beta,
\]
where $L_f$ and $L_g$ are slowly varying at $0$. In this regime, $F_n$ is a prototypical Toeplitz-type quadratic functional arising in frequency-domain inference, including Whittle-type procedures; see Grenander and Szeg\H{o}~\cite{GrenanderSzego1958} and Avram~\cite{Avram1988}.

The asymptotic behavior of $F_n$ is determined by the singularity of $f(\lambda)g(\lambda)$ at the origin: the condition $d<\tfrac12$ yields the standard $\sqrt{n}$-CLT regime, whereas stronger singularities lead to non-central limits. This dichotomy goes back to Fox and Taqqu~\cite{FoxTaqqu1985,FoxTaqqu1987}; see also \cite{GiraitisSurgailis1990,Hosoya1997,Dahlhaus1989} for extensions and statistical applications. A key technical ingredient in the analysis of $F_n$ is the trace approximation problem for products of Toeplitz matrices and operators, which enters moment and cumulant expansions as well as variance asymptotics (see, e.g. \cite{GinovyanSahakyanTaqqu2014,GinovyanTaqqu2022,Ginovyan2025Book}, also \cite{HorvathShao1999} for Whittle estimation and \cite{TaniguchiKakizawa2000} for general background).

Our main contribution is an explicit $1$-Wasserstein bound for the integrated periodogram, without variance standardization, in the Gaussian CLT regime and under a local smoothness assumption on the spectral density and the weight. Our approach relies on a second-chaos representation $F_n=I_2(k_n)$, combined with the Malliavin-Stein method~\cite{NourdinPeccati2009b}. On the second Wiener chaos, normal approximation reduces to controlling the first contraction, equivalently the fourth cumulant. We identify this contraction as the trace of an explicit integral operator and estimate it by combining Schatten-type inequalities with weighted Schur bounds.
\section{Model setting and main result}\label{sec:setup}
Let $(X_t)_{t\in\mathbb{Z}}$ be centered, real-valued, Gaussian and stationary with covariance $r(k)=\E[X_0X_k]$. Assume that $f\in L^1(-\pi,\pi)$ is the spectral density of $(X_t)_{t\in\mathbb{Z}}$, so that (see \cite[Chapter~4]{BrockwellDavis1991})
\begin{equation}\label{eq:cov_spec}
 r(k)=\int_{-\pi}^{\pi} e^{\ii k\lambda}\, f(\lambda)\,\dd\lambda,\qquad k\in\mathbb{Z}.
\end{equation}
For real-valued $(X_t)_{t\in\mathbb Z}$, $f$ is even and nonnegative a.e. Consider a function $g$ that is even a.e.

We recall that for real random variables $F,N$, the 1-Wasserstein distance is
\[d_W(F,N):=\sup_{\Lip(h)\le 1} \abs{\E[h(F)]-\E[h(N)]}.
\]
Let $F_n$ be defined by \eqref{eq:intperiodogram} and let 
\[V_n:=\Var(F_n),\qquad
 \sigma_0^2:=4\pi\int_{-\pi}^{\pi} f(\lambda)^2 g(\lambda)^2\,\dd\lambda, \qquad Z\sim\mathcal{N}(0,\sigma_0^2).\]
\begin{theorem}\label{thm:main}
Let $f,g\in L^{1}(-\pi,\pi)$ such that
\begin{align}\label{ass:LM}
\begin{split}
f(\lambda) &= |\lambda|^{-\alpha} L_f(\lambda),\\
g(\lambda) &= |\lambda|^{-\beta} L_g(\lambda), 
\end{split}
\end{align}
where $L_f$ and $L_g$ are slowly varying functions at $0$. Assume that
\[
-1<\alpha, \beta<1,
\qquad
d:=\alpha+\beta < \frac12 .
\]
Then we get
\[
d_W(F_n,Z)\longrightarrow 0 \qquad \text{as } n\to\infty .
\]
Assume moreover that there exists $C>0$ such that, for all $\lambda,\omega$ satisfying
$|\lambda-\omega|<\tfrac{|\lambda|}{2}$,
\begin{align}\label{ass:meanvalue}
\begin{split}
|f(\omega)-f(\lambda)|
&\le C\, f(\lambda)\,|\lambda|^{-1}|\lambda-\omega|,\\
|g(\omega)-g(\lambda)|
&\le C\, g(\lambda)\,|\lambda|^{-1}|\lambda-\omega|.
\end{split}
\end{align}
Then, for every $\eta>0$, there exists a constant $C_\eta<\infty$ such that, for all $n$ large enough,
\[
d_W(F_n,Z)\le C_\eta\, n^{\,d-\frac12+\eta}.
\]
\end{theorem}
Compared with the  CLTs established by Fox and Taqqu \cite{FoxTaqqu1987} and Ginovyan \cite{Ginovyan1994}, our result yields a quantitative Gaussian approximation for Toeplitz quadratic forms in the short, long, and anti-persistent regimes. This comes at the cost of the additional local smoothness assumption \eqref{ass:meanvalue}, which is stronger than the standard power-law behavior near the origin but is precisely what enables us to obtain an explicit
Wasserstein bound.

On the other hand, our direct contraction bound is not intended to be sharp in the classical $L^p$-$L^q$ integrable regime $\frac1p+\frac1q<\frac14$, or equivalently in our power-law setting
\[
\alpha+\beta_{+}<\frac14,
\qquad \beta_{+}:=\max(\beta,0),\]
where the Toeplitz trace approximation theorem is available. In this case, the fourth cumulant is of order $O(n^{-1})$, while the Kolmogorov distance for the variance-standardized statistic is of order $O(n^{-1/2})$; see Nourdin and Peccati \cite{NourdinPeccati2009}. The main interest of the present theorem is instead to treat the broader long-memory regime in which such trace asymptotics are not directly available. Furthermore, even within the $L^p$-$L^q$ integrable regime, our result remains complementary, as it provides a quantitative control of the variance, which in turn yields a quantitative Wasserstein bound for the non-standardized statistic.

It is also worth noting that Assumption \eqref{ass:meanvalue} is satisfied by standard long-memory models such as FARIMA processes and fractional Gaussian noise. Before turning to the proof, we record a concrete application to frequency-domain inference, namely the asymptotic normality of a (one-dimensional) Whittle estimator.

\section{Application to the Whittle estimator}\label{sec:whittle}
We now illustrate Theorem~\ref{thm:main} for the Whittle estimator in the simplified model
\begin{proposition}
Let
\[
f_\alpha(\lambda)=|\lambda|^{-\alpha}L(\lambda),
\qquad \lambda\in(-\pi,\pi),\quad \alpha\in\mathcal A,
\]
where $\mathcal A=[a,b]\subset[0,1)$, and assume that the true parameter value $\alpha_0$ belongs to $\mathcal A$. Suppose that $L\in \mathcal C^1([-\pi,\pi]\setminus\{0\})$ is positive on $[-\pi,\pi]$ and slowly varying at $0$.

Define the Whittle contrast by
\[
Q_n(\alpha)
=
\frac{1}{2\pi}\int_{-\pi}^{\pi}
\left[
\log f_\alpha(\lambda)+\frac{I_n(\lambda)}{f_\alpha(\lambda)}
\right]\dd\lambda,
\qquad \alpha\in\mathcal A.
\]
Let
\[
\widehat\alpha_n\in \arg\min_{\alpha\in\mathcal A} Q_n(\alpha)
\]
be any measurable minimizer of $Q_n$ over $\mathcal A$. Then
\[
\widehat\alpha_n \xrightarrow{\mathbb P} \alpha_0,
\]
and
\[
\sqrt n\,(\widehat\alpha_n-\alpha_0)
\;\Rightarrow\;
\mathcal N\!\left(0,\frac{4}{\sigma_0^2}\right),
\]
where
\[
\sigma_0^2
=
\frac{1}{2\pi}\int_{-\pi}^{\pi} \bigl(\log|\lambda|\bigr)^2\,\dd\lambda.
\]
\end{proposition}

\begin{proof}

For every \(\alpha\in\mathcal A\),
\[
Q_n(\alpha)
=
\frac{1}{2\pi}\int_{-\pi}^{\pi}
\left(
\ln f_\alpha(\lambda)+\frac{I_n(\lambda)}{f_\alpha(\lambda)}
\right)\,d\lambda .
\]
Since
\[
f_\alpha(\lambda)=|\lambda|^{-\alpha}L(\lambda),
\]
we have
\[
\partial_\alpha \ln f_\alpha(\lambda)=-\ln |\lambda|,
\qquad
\partial_\alpha \big(f_\alpha(\lambda)^{-1}\big)
=\frac{\ln |\lambda|}{f_\alpha(\lambda)}.
\]
Hence
\begin{equation}\label{eq:score-31}
Q_n'(\alpha)
=
\frac{1}{2\pi}\int_{-\pi}^{\pi}
\ln |\lambda|
\left(
\frac{I_n(\lambda)}{f_\alpha(\lambda)}-1
\right)\,d\lambda,
\end{equation}
and
\begin{equation}\label{eq:hessian-31}
Q_n''(\alpha)
=
\frac{1}{2\pi}\int_{-\pi}^{\pi}
\ln^2 |\lambda|\,\frac{I_n(\lambda)}{f_\alpha(\lambda)}\,d\lambda.
\end{equation}
Since \(Q_n\) is continuous on the compact set \(\mathcal A\), it admits at least one minimizer. Moreover,  \(Q_n''(\alpha)>0\) almost surely, because \(I_n(\lambda) \) is not identically zero almost surely, while \( (\ln |\lambda|)^2>0\) for almost every \(\lambda\in(-\pi,\pi)\). Therefore, \(Q_n\) is strictly convex almost surely, and its minimizer is unique. 

The condition on the slowly varying part verifies assumptions (A.2) and (A.4) in Taqqu-Fox \cite{FoxTaqqu1986}, which implies  the consistency. 
\[
\widehat\alpha_n \xrightarrow{\mathbb P} \alpha_0.
\]
We now prove the asymptotic normality. From \eqref{eq:score-31}, evaluated at \(\alpha_0\),
\[
Q_n'(\alpha_0)
=
\int_{-\pi}^{\pi}
\frac{\ln |\lambda|}{2\pi f_{\alpha_0}(\lambda)}
\big(I_n(\lambda)-\mathbb E[I_n(\lambda)]\big)\,d\lambda
+
\int_{-\pi}^{\pi}
\frac{\ln |\lambda|}{2\pi f_{\alpha_0}(\lambda)}
\big(\mathbb E[I_n(\lambda)]-f_{\alpha_0}(\lambda)\big)\,d\lambda .
\]
We apply Theorem~\ref{thm:main} with
\[
g(\lambda)=\frac{\ln |\lambda|}{2\pi f_{\alpha_0}(\lambda)}
=\frac{|\lambda|^{\alpha_0}\ln |\lambda|}{2\pi L(\lambda)}.
\]
This function is even and integrable, and near the origin it is of the form
\[
g(\lambda)=|\lambda|^{-\beta}\widetilde L(\lambda),
\qquad \beta=-\alpha_0,
\]
so that
\[
\alpha_0+\beta=0<\frac12.
\]
Since \(L\in \mathcal{C}^1\) and \(L>0\), the local smoothness assumption of Theorem~\ref{thm:main} is satisfied. Therefore
\[
\sqrt n
\int_{-\pi}^{\pi}
\frac{\ln |\lambda|}{2\pi f_{\alpha_0}(\lambda)}
\big(I_n(\lambda)-\mathbb E[I_n(\lambda)]\big)\,d\lambda
\Rightarrow
\mathcal N\!\left(
0,\sigma_0^2
\right).
\]
with 
\[
\sigma_0^2 = 4\pi\int_{-\pi}^{\pi}
f_{\alpha_0}(\lambda)^2
\left(\frac{\ln |\lambda|}{2\pi f_{\alpha_0}(\lambda)}\right)^2
\,d\lambda
=
\frac{1}{\pi}\int_{-\pi}^{\pi}\ln^2 |\lambda|\,d\lambda.
\]
For the deterministic term, using $\mathbb E[I_n(\lambda)] = \int^{\pi}_{\pi} f_{\alpha_0}(\omega) \Phi_n(\lambda-\omega) \dd \omega$ with $\Phi_n$ denoting the Fejér kernel, we have from Lemma~\ref{lem:diff-convolution}
\begin{align*}
\left|\int_{-\pi}^{\pi}
\frac{\ln |\lambda|}{2\pi f_{\alpha_0}(\lambda)}
\big(\mathbb E[I_n(\lambda)]-f_{\alpha_0}(\lambda)\big)\,d\lambda \right| \longrightarrow 0
\end{align*}
which implies
\begin{equation}\label{eq:score-clt-31}
\sqrt n\,Q_n'(\alpha_0)\Rightarrow \mathcal N(0,\sigma_0^2).
\end{equation}
Next, from \eqref{eq:hessian-31},
\begin{align*}
Q_n''(\alpha_0)
&=
\int_{-\pi}^{\pi}
\frac{\ln^2 |\lambda|}{2\pi f_{\alpha_0}(\lambda)}
I_n(\lambda)\,d\lambda\\
&=
\int_{-\pi}^{\pi}
\frac{\ln^2 |\lambda|}{2\pi f_{\alpha_0}(\lambda)}
\big(I_n(\lambda)-\mathbb E[I_n(\lambda)]\big)\,d\lambda
+
\int_{-\pi}^{\pi}
\frac{\ln^2 |\lambda|}{2\pi f_{\alpha_0}(\lambda)}
\mathbb E[I_n(\lambda)]\,d\lambda.
\end{align*}
Applying Theorem~\ref{thm:main} once more, now with
\[
g(\lambda)=\frac{\ln^2 |\lambda|}{2\pi f_{\alpha_0}(\lambda)},
\]
gives
\[
\int_{-\pi}^{\pi}
\frac{\ln^2 |\lambda|}{2\pi f_{\alpha_0}(\lambda)}
\big(I_n(\lambda)-\mathbb E[I_n(\lambda)]\big)\,d\lambda
=
O_{\mathbb P}(n^{-1/2}),
\]
hence this term converges to \(0\) in probability. Also, applying again Lemma~\ref{lem:diff-convolution}
\[
\int_{-\pi}^{\pi}
\frac{\ln^2 |\lambda|}{2\pi f_{\alpha_0}(\lambda)}
\mathbb E[I_n(\lambda)]\,d\lambda
\longrightarrow
\int_{-\pi}^{\pi}
\frac{\ln^2 |\lambda|}{2\pi f_{\alpha_0}(\lambda)}
f_{\alpha_0}(\lambda)\,d\lambda
=
\frac{1}{2\pi}\int_{-\pi}^{\pi}\ln^2 |\lambda|\,d\lambda
= \dfrac{\sigma_0^2}{2}
\]
Therefore
\begin{equation}\label{eq:hessian-limit-31}
Q_n''(\alpha_0)\xrightarrow{\mathbb P}\dfrac{\sigma_0^2}{2}.
\end{equation}
Now \(Q_n'(\widehat\alpha_n)=0\). By the mean value theorem, there exists \(\widetilde\alpha_n\) between \(\widehat\alpha_n\) and \(\alpha_0\) such that
\[
0
=
Q_n'(\alpha_0)+Q_n''(\widetilde\alpha_n)(\widehat\alpha_n-\alpha_0).
\]
Hence
\[
\sqrt n\,(\widehat\alpha_n-\alpha_0)
=
-\frac{\sqrt n\,Q_n'(\alpha_0)}{Q_n''(\widetilde\alpha_n)}.
\]
Since \(\widehat\alpha_n\to\alpha_0\) in probability, we also have \(\widetilde\alpha_n\to\alpha_0\) in probability, and therefore
\[
Q_n''(\widetilde\alpha_n)\xrightarrow{\mathbb P}\dfrac{\sigma_0^2}{2}
\]
by the continuity argument used above. Combining this with
\eqref{eq:score-clt-31}, \eqref{eq:hessian-limit-31}, and Slutsky's lemma, we obtain
\[
\sqrt n\,(\widehat\alpha_n-\alpha_0)
\Rightarrow
\mathcal N\!\left(0,\frac{4}{\sigma_0^2}\right).
\]
\end{proof}
\begin{remark}
We obtain asymptotic normality without imposing the classical assumptions (see, for instance, \cite{FoxTaqqu1986}) on the second-order $\lambda$-partial derivatives of $f$.
\end{remark}
\section{Proof of the main result}\label{sec:proofs}
\subsection*{Proof strategy}
We realize $(X_t)_{t\in\mathbb Z}$ on an isonormal Gaussian space $(W(h))_{h\in\Hh}$ so that
$X_t=W(h_t)$ and $\langle h_t,h_s\rangle_{\Hh}=r(t-s)$; see \cite[Chapter~1]{Nualart2006}.
Let $I_q$ denote the $q$-th multiple Wiener-It\^o integral. We introduce the normalized Dirichlet kernel and its squared modulus:
\begin{equation*}
D_n(\omega):=\frac{1}{\sqrt{2\pi n}}\sum_{t=1}^n e^{\ii t\omega},
\qquad
\Phi_n(\omega):=|D_n(\omega)|^2.
\end{equation*}
Properties of $D_n$ and $\Phi_n$ are collected in Section~\ref{sec:tech}. We will repeatedly use the basic bound
\begin{align*}
|D_n(\omega)| \leq CH_n(\omega), \qquad H_n(\omega):=\dfrac{\sqrt{n}}{1+n|\omega|} .
\end{align*}

\begin{proposition}[Second-chaos representation]\label{prop:chaos}
There exists $k_n\in\Hh^{\odot 2}$ such that
\begin{equation*}
F_n=I_2(k_n).
\end{equation*}
Moreover, with
\begin{equation}\label{eq:def_Anh}
 A_{n,h}(\lambda,\mu):=\int_{-\pi}^{\pi} h(\omega)\,D_n(\lambda-\omega)\,D_n(\omega-\mu)\,\dd\omega,\qquad h\in L^1(-\pi,\pi),
\end{equation}
we have
\begin{equation}\label{eq:Vn_integral}
 V_n=2\,\norm{k_n}_{\Hh^{\otimes 2}}^2
 =2n\int g(\lambda)g(\mu)\,\left| A_{n,f}(\lambda,\mu) \right|^2\,\dd\lambda\,\dd\mu.
\end{equation}
\end{proposition}

\begin{proof}
From the definition of $d_n$ we write
\[
I_n(\lambda)=\frac{1}{2\pi n}\sum_{t,s=1}^n e^{\ii (t-s)\lambda}X_tX_s.
\]
In the isonormal representation, one has $X_tX_s-\E[X_tX_s]=I_2(h_t\otimes h_s)$; see \cite[Proposition~1.1.3]{Nualart2006}. It follows that
\[
I_n(\lambda)-\E[I_n(\lambda)]
=\frac{1}{2\pi n}\sum_{t,s=1}^n e^{\ii (t-s)\lambda}\,I_2(h_t\otimes h_s).
\]
Plugging this identity into \eqref{eq:intperiodogram} and using the linearity of $I_2$ gives the representation $F_n=I_2(k_n)$ with
\[
 k_n=\frac{1}{2\pi\sqrt{n}}\int_{-\pi}^{\pi} g(\lambda)\sum_{t,s=1}^n e^{\ii (t-s)\lambda}\,h_t\otimes h_s\,\dd\lambda.
\]
Finally, $\Var(I_2(k_n))=2\|k_n\|^2$ \cite[Proposition~1.1.2]{Nualart2006}.
Expanding $\|k_n\|^2$ by Fubini and using $\langle h_t\otimes h_s, h_u\otimes h_v\rangle = r(t-u)r(s-v)$
followed by \eqref{eq:cov_spec} and an elementary rearrangement yields \eqref{eq:Vn_integral} with $A_{n,f}$ as in \eqref{eq:def_Anh}.
\end{proof}
A general Malliavin-Stein inequality for Wasserstein distance \cite[Theorem~3.1 and Theorem~3.2]{NourdinPeccati2009b} yields,
for $F=I_2(k)$
\begin{align}
\label{lem:MS_second}
d_W(F,N)\le C\sqrt{\left( 1 - \norm{k}^2_{\Hh^{\otimes 2}} \right)^2 + \norm{k\ots k}_{\Hh^{\otimes 2}}^2 },   
\end{align}
with $N\sim\mathcal{N}(0,1)$. We bound in the following steps the variance and the first contraction.
\subsection{Variance control}\label{subsec:var}
\begin{proposition}
\label{prop:var}
Under Assumption~\eqref{ass:LM}, we have
\[V_n \longrightarrow \sigma_0^2.\]
In addition, if Assumption~\eqref{ass:meanvalue} is verified, then for every $\eta>0$, there exists $C_\eta<\infty$ such that for all $n$ large enough,
\begin{equation}\label{eq:var_rate}
\abs{V_n-\sigma_0^2}\le C_\eta\,n^{2d-1+\eta}.
\end{equation}
\end{proposition}

\begin{proof}
From the symmetry of $\lambda$ and $\mu$, it is sufficient to limit on the region $\mathcal{R}_1 = \lbrace|g(\mu)| < |g(\lambda)|\rbrace$.\\\\
We start from \eqref{eq:Vn_integral}:\begin{equation*}
 V_n=2\,\norm{k_n}_{\Hh^{\otimes 2}}^2
 =2n\int g(\lambda)g(\mu)\,\left| A_{n,f}(\lambda,\mu) \right|^2\,\dd\lambda\,\dd\mu.
\end{equation*}
Using the definitions of $D_n$ and $\Phi_n$ and the convolution identity
$\int_{-\pi}^{\pi} D_n(\lambda-\omega)D_n(\omega-\mu)\dd\omega=\sqrt{\frac{2\pi}{n}}\,D_n(\lambda-\mu)$
(see Lemma~\ref{lem:dirichlet}), we decompose
\begin{equation*}
A_{n,f}(\lambda,\mu)=\sqrt{\frac{2\pi}{n}}\,f(\lambda)D_n(\lambda-\mu)+E_{n,f}(\lambda,\mu),
\end{equation*}
where
\begin{equation*}
E_{n,f}(\lambda,\mu):=\int_{-\pi}^{\pi} \big(f(\omega)-f(\lambda)\big)D_n(\lambda-\omega)D_n(\omega-\mu)\,\dd\omega.
\end{equation*}
Expanding $n|A_{n,f}|^2$ gives
\begin{equation*}
\begin{aligned}n\abs{A_{n,f}(\lambda,\mu)}^2
 &=2\pi f(\lambda)^2\Phi_n(\lambda-\mu)
 +n\abs{E_{n,f}(\lambda,\mu)}^2 +2\sqrt{2\pi n}\,f(\lambda)\,\Re\Big(D_n(\lambda-\mu)\overline{E_{n,f}(\lambda,\mu)}\Big).
\end{aligned}
\end{equation*}
Plugging the expansion above into \eqref{eq:Vn_integral} yields
\begin{equation*}
V_n=4\pi M_n +2P_n +4\sqrt{2\pi}\,\Re(R_n),
\end{equation*}
where
\begin{align*}
 M_n&:=\int_{\mathcal{R}_1} g(\lambda)g(\mu)\,f(\lambda)^2\Phi_n(\lambda-\mu)\,\dd\lambda\,\dd\mu,\\
 P_n&:=n\int_{\mathcal{R}_1} g(\lambda)g(\mu)\,\abs{E_{n,f}(\lambda,\mu)}^2\,\dd\lambda\,\dd\mu,\\
 R_n&:=\sqrt{n}\int_{\mathcal{R}_1} g(\lambda)g(\mu)\,f(\lambda)\,D_n(\lambda-\mu)\overline{E_{n,f}(\lambda,\mu)}\,\dd\lambda\,\dd\mu.
\end{align*}
\noindent We first estimate $M_n$.
We write
\begin{align*}
\left| M_n-\int_{\mathcal{R}_1} f(\lambda)^2g(\lambda)^2\,\dd\lambda \right|
\leq C\int_{\mathcal{R}_1} f(\lambda)^2|g(\lambda) (g(\mu)-g(\lambda))|H^2_n(\lambda-\mu)\,\dd\mu \dd\lambda.
\end{align*}
Applying Lemma~\ref{lem:diff-convolution}, we get the desired asymptotics for $M_n$.\\
For $P_n$, we first bound $\displaystyle |P_n| \leq n\int_{\mathcal{R}_1} g^2(\lambda)\abs{E_{n,f} (\lambda,\mu)}^2\,\dd\lambda \dd\mu$ and then apply Lemma~\ref{lem:Enf-L2} to obtain
\begin{align*}
|P_n| &\leq C \int g^2(\lambda)(f(\lambda)-f(\omega))^2 \Phi_n(\lambda-\omega)\dd\omega \dd\lambda \\
& \leq C \int g^2(\lambda)\left| f^2(\lambda)-f^2(\omega)\right| \Phi_n(\lambda-\omega)\dd\omega \dd\lambda.
\end{align*}
The desired asymptotics for $P_n$ then follow from another application of Lemma~\ref{lem:diff-convolution}.\\
Finally, using the Cauchy-Schwarz inequality, one gets $|R_n| \leq \sqrt{M_n|P_n|}$, then $R_n \rightarrow 0$. To derive the rate, we write $R_n$ in explicit form.
\begin{align*}
|R_n| &\leq \sqrt{n}\int |g(\lambda)g(\mu)f(\lambda) \big(f(\omega)-f(\lambda)\big)D_n(\omega-\lambda) D_n(\lambda-\mu)D_n(\mu-\omega)|\dd\mu \dd\omega\dd\lambda\\
&\leq \sqrt{n}\int |g^2(\lambda)f(\lambda) \big(f(\omega)-f(\lambda)\big)D_n(\omega-\lambda) D_n(\lambda-\mu)D_n(\mu-\omega)|\dd\mu \dd\omega\dd\lambda \\
&\leq C_{\eta} n^{\eta}\int |g^2(\lambda)f(\lambda) \big(f(\omega)-f(\lambda)\big)|H_n^2(\lambda-\omega)\dd\omega \dd\lambda \\
&\leq C_{\eta} n^{2d + \eta-1},
\end{align*}
where the third inequality comes from Lemma~\ref{lem:dirichlet}. This concludes the proposition.
\end{proof}
\subsection{Contraction / fourth cumulant}\label{subsec:contr}
For a random variable $F$, denote by
\[\kappa_4(F):=\E[F^4]-3(\E[F^2])^2\]
its fourth cumulant. If $F=I_2(k)$, then (see, e.g., \cite[Proposition~5.2.7]{NourdinPeccati2012})
\begin{equation*}
\kappa_4(F)=48\,\norm{k\ots k}_{\Hh^{\otimes2}}^2.
\end{equation*}

\begin{proposition}[Contraction control]\label{prop:contr}
Under Assumption~\eqref{ass:LM}, we have
\[\norm{k_n\ots k_n}_{\Hh^{\otimes 2}}^2 \longrightarrow 0.\]
In addition, if Assumption~\eqref{ass:meanvalue} is verified, then for every $\eta>0$, there exists $C_\eta<\infty$ such that for all $n$ large enough,
\begin{equation*}
\norm{k_n\ots k_n}_{\Hh^{\otimes 2}}^2 \le C_\eta\,n^{2d-1+\eta}.
\end{equation*}
\end{proposition}

\begin{proof}
We start by expressing the contraction through a trace identity.
Start from the expression of $k_n$ in the proof of Proposition~\ref{prop:chaos}.
Using the bilinearity of $\ots$ and the identity
$(a\otimes b)\ots(c\otimes d)=\langle b,c\rangle_{\Hh}\,a\otimes d$,
we obtain
\begin{align*}
 k_n\ots k_n
 &=\frac{1}{(2\pi)^2n}\int g(\lambda)g(\mu)
 \sum_{t,s,u,v=1}^n e^{\ii(t-s)\lambda}e^{\ii(u-v)\mu}\,(h_t\otimes h_s)\ots(h_u\otimes h_v)\dd \lambda\dd \mu\\
 &=\frac{1}{(2\pi)^2n}\int g(\lambda)g(\mu)
 \sum_{t,s,u,v=1}^n e^{\ii(t-s)\lambda}e^{\ii(u-v)\mu}\,r(s-u)\,h_t\otimes h_v\dd \lambda\dd \mu.
\end{align*}
Taking the $\Hh^{\otimes2}$-norm and using
$\langle h_t\otimes h_v, h_{t'}\otimes h_{v'}\rangle=r(t-t')r(v-v')$ gives
\begin{align*}
\norm{k_n\ots k_n}_{\Hh^{\otimes2}}^2
&=\frac{1}{(2\pi)^4n^2}\int\prod_{j=1}^4 g(\lambda_j)
\sum_{\substack{t,s,u,v\\t',s',u',v'=1}}^n
 e^{\ii(t-s)\lambda_1}e^{\ii(u-v)\lambda_2}e^{-\ii(t'-s')\lambda_3}e^{-\ii(u'-v')\lambda_4}\\
&\hspace{2.1cm}\times r(s-u)\,r(s'-u')\,r(t-t')\,r(v-v')\dd \lambda_1\dd \lambda_2\dd \lambda_3\dd \lambda_4.
\end{align*}
Insert the spectral representation \eqref{eq:cov_spec} for each covariance factor and rearrange sums.
A direct (but standard) factorization yields
\begin{equation*}
\norm{k_n\ots k_n}_{\Hh^{\otimes2}}^2
= n^2\int_{(-\pi,\pi)^4}\Big(\prod_{j=1}^4 g(\lambda_j)\Big)
A_{n,f}(\lambda_1,\lambda_2)A_{n,f}(\lambda_2,\lambda_3)A_{n,f}(\lambda_3,\lambda_4)A_{n,f}(\lambda_4,\lambda_1)
\dd\lambda_1\cdots\dd\lambda_4.
\end{equation*}
Define the integral operator $L_n$ on $L^2(-\pi,\pi)$ with kernel
\[\ell_n(\lambda,\mu):=\sqrt{g(\lambda)}\,A_{n,f}(\lambda,\mu)\,\sqrt{g(\mu)}.\]
Then $L_n$ is Hilbert-Schmidt and
\begin{equation*}
\norm{k_n\ots k_n}_{\Hh^{\otimes2}}^2 = n^2\,\Tr(L_n^4).
\end{equation*}
Indeed,
$\Tr(L_n^4)=\int \ell_n(\lambda_1,\lambda_2)\ell_n(\lambda_2,\lambda_3)\ell_n(\lambda_3,\lambda_4)\ell_n(\lambda_4,\lambda_1)\dd\lambda_{1:4}$,
which matches the integrand in the cycle integral representation above after the cyclic cancelation of the square roots. We then apply a Schatten-type trace inequality (see, e.g., \cite{Simon2005})
\begin{equation*}
|\Tr(L_n^4)|\le \|L_n\|_{\op}^2\,\|L_n\|_{\HS}^2.
\end{equation*}
We first bound the Hilbert-Schmidt norm of $L_n$ using $V_n$
\begin{align*}
n\|L_n\|_{\HS}^2=\iint g(\lambda)g(\mu)\,|A_{n,f}(\lambda,\mu)|^2\dd\lambda\dd\mu \leq V_n \leq C,
\end{align*}
hence $n\|L_n\|_{\HS}^2 \leq C$ for $n$ large enough. Therefore
\[
\norm{k_n\ots k_n}^2
= n^2|\Tr(L_n^4)|
\le n^2\|L_n\|_{\op}^2\|L_n\|_{\HS}^2 \le Cn\|L_n\|_{\op}^2.
\]
It remains to bound the operator norm term. We decompose
\[
L_n= L_n^{(0)}+ L_n^{(1)},
\]
where $L_n^{(0)}$ and $L_n^{(1)}$ have kernels $\ell_n^{(0)}$ and $\ell_n^{(1)}$ respectively 
\[
\ell_n^{(0)}(\lambda,\mu)=\sqrt{\frac{2\pi}{n}}\sqrt{f(\lambda)f(\mu) g(\lambda)g(\mu)}\,D_n(\lambda-\mu), \qquad
\ell_n^{(1)}(\lambda,\mu)=\sqrt{g(\lambda)g(\mu)}\,\widetilde E_{n,f}(\lambda,\mu),\]
with
\[
\widetilde E_{n,f}(\lambda,\mu):= E_{n,f}(\lambda,\mu)
+\sqrt{\frac{2\pi}{n}}\,\sqrt{f(\lambda)}\big(\sqrt{f(\lambda)}-\sqrt{f(\mu)}\big)D_n(\lambda-\mu).
\]
For $\| L_n^{(1)}\|_{\op}$, since the operator norm is bounded by the Hilbert-Schmidt norm, we have
\begin{align*}
\|L_n^{(1)}\|_{\op}^2\le \| L_n^{(1)}\|_{\HS}^2
&=\int g(\lambda)g(\mu)\,|\widetilde E_{n,f}(\lambda,\mu)|^2\dd\lambda\dd\mu.
\end{align*}
Again, by symmetry between $\lambda$ and $\mu$, it suffices to restrict attention to the region $\mathcal{R}_1 = \lbrace|g(\mu)| < |g(\lambda)|\rbrace$. Note that
\begin{align*}
\int_{\mathcal{R}_1 }  g(\lambda)g(\mu)\,|\widetilde E_{n,f}(\lambda,\mu)|^2\dd\lambda\dd\mu &\leq 2\int_{\mathcal{R}_1 } |g(\lambda)g(\mu)E_{n,f}(\lambda,\mu)|^2\dd\lambda\dd\mu \\
& \quad + \frac{4\pi}{n}\int_{\mathcal{R}_1 }  |g(\lambda)g(\mu)|f(\lambda)\big(\sqrt{f(\lambda)}-\sqrt{f(\mu)}\big)^2|D_n(\lambda-\mu)|^2\dd\lambda\dd\mu\\
&\leq Cn^{-1}|P_n| + C{n}^{-1}\int g^2(\lambda)f(\lambda)\big|f(\lambda) - f(\mu)\big|H^2_n(\lambda-\mu)\dd\lambda\dd\mu.
\end{align*}
From the estimate of $P_n$ and Lemma~\ref{lem:diff-convolution}, we conclude that
\begin{align*}
n\|L_n^{(1)}\|_{\op}^2 \longrightarrow 0, \qquad \|L_n^{(1)}\|_{\op}^2 \leq C_\eta n^{2d + \eta - 2} \quad \text{under \eqref{ass:meanvalue}.}  
\end{align*}
Finally, for $\| L_n^{(0)}\|_{\op}$, it can be shown that from Lemma~\ref{lem:one_denom} one has
\begin{align*}
\frac{1}{p(\mu)}\int |\ell_n^{(0)}(\lambda,\mu)|\,p(\lambda)\dd\lambda \leq \int f(\lambda)|g(\lambda)D_n(\lambda-\mu)| \dd\lambda \leq C_\eta\,n^{d-1+\eta}
\end{align*}
with $p = \sqrt{f|g|}$. 
Apply the Schur test (Lemma~\ref{lem:schur}) with weight $p = \sqrt{f|g|}$, we get $\| L_n^{(0)}\|_{\op} \leq C_\eta\,n^{d-1+\eta}$. It then follows that
\begin{align*}
n\|L_n\|_{\op}^2 \longrightarrow 0, \qquad \|L_n \|_{\op}^2 \leq C_\eta n^{2d + \eta - 2} \quad \text{under \eqref{ass:meanvalue},}  
\end{align*}
which finally implies the lemma
\begin{align*}
\norm{k_n\ots k_n}^2 \longrightarrow 0, \qquad \norm{k_n\ots k_n}^2 \leq C_\eta n^{2d + \eta - 1} \quad \text{under \eqref{ass:meanvalue}.}
\end{align*}
\end{proof}
\subsection*{Proof of Theorem~\ref{thm:main}}
\begin{proof}
Using \eqref{lem:MS_second}, we get
\begin{align*}
 d_W(F_n,Z) = \sigma_0\, d_W\left( \dfrac{F_n}{\sigma_0}, N \right) \leq C\sqrt{ \left( \sigma_0 - V_n \right)^2 + \norm{k_n\ots k_n}_{\Hh^{\otimes2}}^2}.
\end{align*}
From Proposition~\ref{prop:var} and Proposition~\ref{prop:contr}, we deduce
\[d_W(F_n,Z) \longrightarrow 0\]
and under assumption~\eqref{ass:meanvalue}
\begin{align*}
d_W(F_n,Z) \le C_\eta\,n^{d-\tfrac{1}{2}+\eta}.
\end{align*}
\end{proof}
\section{Technical lemmas}\label{sec:tech}
\begin{lemma}[Dirichlet bounds]\label{lem:dirichlet}
There exists $C<\infty$ such that for all $n\ge1$ and $\lambda\in(-\pi,\pi)$, we have
\[
\abs{D_n(\lambda)}\le C\min\Big(\sqrt n,\frac{1}{\sqrt n\,\abs{\lambda}}\Big)
\le C H_n(\lambda),
\qquad
\Phi_n(\lambda):=\abs{D_n(\lambda)}^2 \le H_n^2(\lambda).
\]
Moreover,
\[
\int_{-\pi}^{\pi}\Phi_n(\lambda)\dd\lambda=1,
\qquad
\int_{-\pi}^{\pi}D_n(\lambda-\omega)D_n(\omega-\mu)\dd\omega
=\sqrt{\frac{2\pi}{n}}\,D_n(\lambda-\mu),
\]
\begin{align*}
\int_{-\pi}^{\pi}|D_n(\lambda-\omega)D_n(\omega-\mu)|\dd\omega \leq C\int_{-\pi}^{\pi}|H_n(\lambda-\omega)H_n(\omega-\mu)|\dd\omega \leq C_{\eta}n^{\eta-1/2}H_n(\lambda-\mu).
\end{align*}
\end{lemma}
\begin{proof}
Using the explicit formula
$D_n(\lambda)=e^{\ii(n+1)\lambda/2}\,\frac{\sin(n\lambda/2)}{\sqrt{2\pi n}\,\sin(\lambda/2)}$, the bounds are standard; see
\cite[Chapter~II]{Katznelson2004} or \cite[Chapter~I]{Zygmund2002}.
The identity $\int\Phi_n=1$ follows from an explicit Fourier-series computation.
The convolution identity follows by expanding the sums defining $D_n$ and integrating termwise.
\end{proof}
\begin{lemma}[Potter bound at $0$]\label{lem:potter}
For any slowly varying $L$ at $0$ and any $\varepsilon>0$, there exist $C_\varepsilon$ and $\delta>0$ such that
for all $0<|\lambda|,|\omega|\le\delta$,
\[
\frac{L(\omega)}{L(\lambda)}
\le C_\varepsilon\Big(\big(\tfrac{|\omega|}{|\lambda|}\big)^{\varepsilon} \vee \big(\tfrac{|\omega|}{|\lambda|}\big)^{-\varepsilon}\Big).
\]
\end{lemma}
\begin{proof}
This is Potter's bound; see \cite[Theorem~1.5.6]{BinghamGoldieTeugels1987}.
\end{proof}

\begin{lemma}\label{lem:one_denom}
Let $0 \leq \gamma < 1$ and $h(\lambda) = |\lambda|^{-\gamma}L(\lambda)$. Then
\begin{align*}
\int |h(\lambda) D_n(\lambda-\mu)| \dd\lambda \leq C_{\eta}n^{\gamma+\eta-1/2}.
\end{align*}
\end{lemma}
\begin{proof}
We first get
\begin{align*}
\int_{|\lambda|\le 2/n}|h(\lambda) D_n(\lambda-\mu)| \dd\lambda
\leq \sqrt{n} \int_{|\lambda|\le 2/n}|h(\lambda)| \dd\lambda \leq  C_{\eta}n^{\gamma+\eta-1/2}.
\end{align*}
Consider now the region $|\lambda|>2/n$. We bound
\begin{align*}
\int_{|\lambda|> 2/n}|h(\lambda) D_n(\lambda-\mu)| \dd\lambda
\leq  C_{\eta}n^{\eta}\int_{|\lambda|> 2/n}|\lambda|^{-\gamma} |D_n(\lambda-\mu)| \dd\lambda.
\end{align*}
For the region $|\lambda-\mu|\le |\lambda|/2$, which implies also $|\lambda-\mu|\le |\mu|$ and $|\mu| > |\lambda| /2 > 1/n$. In addition, as $|\lambda|^{-\gamma} \leq C(|\lambda-\mu|^{-\gamma} + |\mu|^{-\gamma})$, we need to bound $\int_{|\lambda-\mu| < |\mu|}|(\lambda-\mu)|^{-\gamma} |D_n(\lambda-\mu)| \dd\lambda$, equivalently $\int_{|\omega| < |\mu|}|\omega|^{-\gamma} |D_n(\omega)| \dd\omega$.
We have
\begin{align*}
\int_{|\omega| < |\mu|}|\omega|^{-\gamma} |D_n(\omega)| \dd\omega \leq \sqrt{n}\int_{|\omega| < 1/n}|\omega|^{-\gamma} \dd\omega + C\tfrac{1}{\sqrt{n}} \left| \int_{1/n}^{\mu}|\omega|^{-\gamma-1} \dd\omega \right| &\leq C_{\eta}n^{\eta - 1/2}\left(n^{\gamma} + \mu^{-\gamma} \right) \\
&\leq C_{\eta}n^{\eta + \gamma - 1/2}.
\end{align*}
We study finally the case $|\lambda-\mu|>|\lambda|/2$, which implies $|D_n(\lambda-\mu)| \leq C\lambda^{-1}$. Then, one has
\begin{align*}
\int_{\substack{|\lambda|> 2/n \\ |\lambda-\mu|>|\lambda|}}|\lambda|^{-\gamma} |D_n(\lambda-\mu)| \dd\lambda \leq \int_{|\lambda|> 2/n}|\lambda|^{-\gamma-1} \dd\lambda \leq n^{\gamma}.
\end{align*}
\end{proof}
\begin{lemma}[Weighted Schur test]\label{lem:schur}
Let $K$ be an integral operator on $L^2$ with kernel $k(\lambda,\mu)$.
If there exists a positive weight $p$ and $M<\infty$ such that
\[
\sup_{\lambda}\frac1{p(\lambda)}\int |k(\lambda,\mu)|p(\mu)\dd\mu\le M,
\qquad
\sup_{\mu}\frac1{p(\mu)}\int |k(\lambda,\mu)|p(\lambda)\dd\lambda\le M,
\]
then $\norm{K}_{\op}\le M$.
\end{lemma}

\begin{proof}
This is the classical Schur test; see, e.g., \cite[Theorem~5.2]{HalmosSunder1978}.
For completeness, let $(Tf)(\lambda)=\int k(\lambda,\mu)f(\mu)\dd\mu$.
By Cauchy-Schwarz,
\[
|Tf(\lambda)|^2\le \Big(\int |k(\lambda,\mu)|p(\mu)\dd\mu\Big)\Big(\int |k(\lambda,\mu)|\frac{|f(\mu)|^2}{p(\mu)}\dd\mu\Big)
\le Mp(\lambda)\int |k(\lambda,\mu)|\frac{|f(\mu)|^2}{p(\mu)}\dd\mu.
\]
Integrate in $\lambda$, use Fubini and the second Schur bound to get $\|Tf\|_2^2\le M^2\|f\|_2^2$.
\end{proof}

\begin{remark}[Trace-class and Hilbert-Schmidt operators]
An integral operator $K$ on $L^2$ with kernel $k\in L^2$ is Hilbert-Schmidt, with norm
$\|K\|_{\HS}^2=\iint |k(\lambda,\mu)|^2\dd\lambda\dd\mu$.
If additionally the singular values of $K$ are summable, then $K$ is trace-class and
$\Tr(K)=\sum_j\langle Ke_j,e_j\rangle$ for any orthonormal basis $(e_j)$.
For a Hilbert-Schmidt operator one always has $\|K\|_{\op}\le \|K\|_{\HS}$.
See \cite{Simon2005} for background on trace ideals.
\end{remark}
\begin{lemma}\label{lem:diff-convolution}
Denote
\[
\Delta_n:=\int_{-\pi}^{\pi}\int_{-\pi}^{\pi}
f(\lambda)\,|g(\lambda)-g(\mu)|\,H_n^2(\lambda-\mu)\,\dd\lambda\,\dd\mu,
\]
where $f(\lambda)=|\lambda|^{-\alpha}L_f(\lambda)$ and $g(\lambda)=|\lambda|^{-\beta}L_g(\lambda)$ with $\alpha, \beta < 1$. Assume that $\alpha+\beta<1$, then
\[
\Delta_n\longrightarrow 0.
\]
In addition, if \(g\) satisfies \eqref{ass:meanvalue} and \(\alpha+\beta\ge 0\), then for every \(\eta>0\),
\[
\Delta_n\le C_\eta n^{\alpha+\beta+\eta-1}.
\]
\end{lemma}
\begin{proof}
Set
\[
\Phi(\omega):=\int_{-\pi}^{\pi} f(\lambda)\,|g(\lambda)-g(\lambda-\omega)|\,\dd\lambda .
\]
After the change of variables \(\omega=\lambda-\mu\), the periodicity of \(f\) and \(g\) yields
\[
\Delta_n=\int_{-\pi}^{\pi}\Phi(\omega)\,H_n^2(\omega)\,\dd\omega .
\]
It therefore remains to analyze the behavior of \(\Phi(\omega)\) as \(\omega\to 0\).

We first prove the qualitative statement. Let \(\varepsilon>0\) be small enough so that
\[
\alpha+\beta+\varepsilon<1.
\]
By Potter's bounds, possibly after reducing \(\varepsilon\),
\[
f(\lambda)\le C_\varepsilon |\lambda|^{-\alpha-\varepsilon/2},
\qquad
|g(\lambda)|\le C_\varepsilon |\lambda|^{-\beta-\varepsilon/2}
\]
for all sufficiently small \(|\lambda|\). We then decompose
\[
\Phi(\omega)=A(\omega)+B(\omega),
\]
where
\[
A(\omega):=\int_{|\lambda|\le 2|\omega|} f(\lambda)\,|g(\lambda)-g(\lambda-\omega)|\,\dd\lambda
\]
and
\[
B(\omega):=\int_{|\lambda|>2|\omega|} f(\lambda)\,|g(\lambda)-g(\lambda-\omega)|\,\dd\lambda .
\]

The term \(A(\omega)\) is treated in the same way in both parts of the lemma. Using
\[
|g(\lambda)-g(\lambda-\omega)|\le |g(\lambda)|+|g(\lambda-\omega)|,
\]
we obtain
\[
A(\omega)\le \int_{|\lambda|\le 2|\omega|} f(\lambda)|g(\lambda)|\,\dd\lambda
+\int_{|\lambda|\le 2|\omega|} f(\lambda)|g(\lambda-\omega)|\,\dd\lambda .
\]
The first integral is bounded by
\[
C_\varepsilon \int_{|\lambda|\le 2|\omega|} |\lambda|^{-\alpha-\beta-\varepsilon}\,\dd\lambda
\le C_\varepsilon |\omega|^{1-\alpha-\beta-\varepsilon}.
\]
For the second one, the change of variables \(\lambda=\omega u\) gives
\[
\int_{|\lambda|\le 2|\omega|} |\lambda|^{-\alpha-\varepsilon/2}|\lambda-\omega|^{-\beta-\varepsilon/2}\,\dd\lambda
=
|\omega|^{1-\alpha-\beta-\varepsilon}
\int_{|u|\le 2} |u|^{-\alpha-\varepsilon/2}|u-1|^{-\beta-\varepsilon/2}\,\dd u,
\]
and the last integral is finite since \(\alpha<1\) and \(\beta<1\). Consequently,
\[
A(\omega)\le C_\varepsilon |\omega|^{1-\alpha-\beta-\varepsilon}.
\]

We next consider \(B(\omega)\). Fix \(\delta>0\) and write
\[
B(\omega)=B_1(\omega,\delta)+B_2(\omega,\delta),
\]
where
\[
B_1(\omega,\delta):=\int_{2|\omega|<|\lambda|\le \delta}
f(\lambda)\,|g(\lambda)-g(\lambda-\omega)|\,\dd\lambda
\]
and
\[
B_2(\omega,\delta):=\int_{|\lambda|>\delta}
f(\lambda)\,|g(\lambda)-g(\lambda-\omega)|\,\dd\lambda .
\]

On the region \(2|\omega|<|\lambda|\le \delta\), we have \(|\lambda-\omega|\asymp |\lambda|\). It follows that
\[
|g(\lambda)|+|g(\lambda-\omega)|\le C_\varepsilon |\lambda|^{-\beta-\varepsilon/2},
\]
and therefore
\[
B_1(\omega,\delta)
\le C_\varepsilon \int_{|\lambda|\le \delta} |\lambda|^{-\alpha-\beta-\varepsilon}\,\dd\lambda
\le C_\varepsilon \delta^{\,1-\alpha-\beta-\varepsilon}.
\]
This estimate is uniform over \(|\omega|<\delta/2\), and tends to \(0\) as \(\delta\downarrow0\).

On the region \(|\lambda|>\delta\), the function \(f\) is bounded. Thus,
\[
B_2(\omega,\delta)\le C_\delta \int_{-\pi}^{\pi} |g(\lambda)-g(\lambda-\omega)|\,\dd\lambda .
\]
Since \(\beta<1\), we have \(g\in L^1([-\pi,\pi])\), and translations are continuous in \(L^1\). Hence
\[
B_2(\omega,\delta)\xrightarrow[\omega\to0]{}0
\qquad\text{for every fixed }\delta>0.
\]

Combining the estimates for \(A(\omega)\), \(B_1(\omega,\delta)\), and \(B_2(\omega,\delta)\), we conclude that
\[
\Phi(\omega)\xrightarrow[\omega\to0]{}0.
\]
It remains to establish a uniform bound for \(\Phi\). The estimate for \(A(\omega)\) already shows that \(A(\omega)\le C\). As for \(B(\omega)\), the relation \(|\lambda-\omega|\asymp |\lambda|\) on \(\{|\lambda|>2|\omega|\}\) implies
\[
B(\omega)\le C_\varepsilon \int_{|\lambda|>2|\omega|} |\lambda|^{-\alpha-\beta-\varepsilon}\,\dd\lambda \le C_\varepsilon,
\]
since \(\alpha+\beta+\varepsilon<1\). Therefore
\[
\sup_{\omega\in[-\pi,\pi]}\Phi(\omega)<\infty.
\]

Since \(H_n^2\) is an approximate identity, we have \(\int_{-\pi}^{\pi}H_n^2(\omega)\,\dd\omega\le C\), and
\[
\int_{|\omega|>\delta} H_n^2(\omega)\,\dd\omega
\le C n^{-1}\int_{|\omega|>\delta} |\omega|^{-2}\,\dd\omega
\le \frac{C}{n\delta}.
\]
It follows that, for every \(\delta>0\),
\[
\Delta_n
\le
\sup_{|\omega|\le \delta}\Phi(\omega)\int_{-\pi}^{\pi}H_n^2(\omega)\,\dd\omega
+
\|\Phi\|_\infty \int_{|\omega|>\delta}H_n^2(\omega)\,\dd\omega .
\]
The first term can be made arbitrarily small by taking \(\delta\) sufficiently small, while the second one then converges to \(0\) as \(n\to\infty\). This proves that \(\Delta_n\to0\).

We next establish the quantitative estimate under Assumption~\eqref{ass:meanvalue}. Let \(\eta>0\). Applying Potter's bounds with exponent \(\eta/2\), we obtain
\[
f(\lambda)\le C_\eta |\lambda|^{-\alpha-\eta/2},
\qquad
|g(\lambda)|\le C_\eta |\lambda|^{-\beta-\eta/2}
\]
for all sufficiently small \(|\lambda|\). As above, this yields
\[
A(\omega)\le C_\eta |\omega|^{1-\alpha-\beta-\eta}.
\]

We now estimate \(B(\omega)\). Since \(|\lambda|>2|\omega|\), we have \(|\omega|<|\lambda|/2\), so Assumption~\eqref{ass:meanvalue} applies to \(\lambda\) and \(\lambda-\omega\), giving
\[
|g(\lambda)-g(\lambda-\omega)|
\le C g(\lambda)|\lambda|^{-1}|\omega|.
\]
Therefore,
\[
B(\omega)
\le C |\omega|\int_{|\lambda|>2|\omega|} f(\lambda)g(\lambda)|\lambda|^{-1}\,\dd\lambda
\le C_\eta |\omega|\int_{|\lambda|>2|\omega|} |\lambda|^{-1-\alpha-\beta-\eta}\,\dd\lambda .
\]
Since \(\alpha+\beta\ge0\), it follows that
\[
B(\omega)\le C_\eta |\omega|^{1-\alpha-\beta-\eta}.
\]
Consequently,
\[
\Phi(\omega)\le C_\eta |\omega|^{1-\alpha-\beta-\eta}.
\]
Using the bound \(H_n^2(\omega)\le C\min(n,n^{-1}|\omega|^{-2})\), we get
\[
\Delta_n
\le C_\eta \int_{-\pi}^{\pi} |\omega|^{1-\alpha-\beta-\eta} H_n^2(\omega)\,\dd\omega
\]
and thus
\[
\Delta_n
\le
C_\eta n\int_0^{1/n} \omega^{1-\alpha-\beta-\eta}\,\dd\omega
+
C_\eta n^{-1}\int_{1/n}^{\pi}\omega^{-1-\alpha-\beta-\eta}\,\dd\omega.
\]
Both terms are of order \(n^{\alpha+\beta+\eta-1}\), and therefore
\[
\Delta_n\le C_\eta n^{\alpha+\beta+\eta-1}.
\]
This completes the proof.
\end{proof}
\begin{lemma}\label{lem:Enf-L2}
For \(\lambda,\mu\in[-\pi,\pi]\), recall that
\[
E_{n,f}(\lambda,\mu)
:=
\int_{-\pi}^{\pi}
\bigl(f(\omega)-f(\lambda)\bigr)\,
D_n(\lambda-\omega)\,D_n(\omega-\mu)\,\dd\omega.
\]
Then,
\[
\int_{-\pi}^{\pi} |E_{n,f}(\lambda,\mu)|^2\,\dd\mu
\le
\frac{2\pi}{n}\int_{-\pi}^{\pi}
|f(\omega)-f(\lambda)|^2\,\Phi_n(\lambda-\omega)\,\dd\omega .
\]
\end{lemma}

\begin{proof}
Set
\[
\psi_\lambda(\omega):=
\bigl(f(\omega)-f(\lambda)\bigr)D_n(\lambda-\omega).
\]
Then
\begin{align*}
E_{n,f}(\lambda,\mu)
&=
\int_{-\pi}^{\pi}\psi_\lambda(\omega)D_n(\omega-\mu)\,\dd\omega \\
&=
\frac{1}{\sqrt{2\pi n}}
\sum_{t=1}^n
\left(
\int_{-\pi}^{\pi}\psi_\lambda(\omega)e^{\ii t\omega}\,\dd\omega
\right)e^{-\ii t\mu} \\
&=
\frac{1}{\sqrt{2\pi n}}\sum_{t=1}^n a_t(\lambda)e^{-\ii t\mu},
\end{align*}
where
\[
a_t(\lambda):=\int_{-\pi}^{\pi}\psi_\lambda(\omega)e^{\ii t\omega}\,\dd\omega.
\]
Therefore,
\begin{align*}
\int_{-\pi}^{\pi}|E_{n,f}(\lambda,\mu)|^2\,\dd\mu
&=
\frac{1}{2\pi n}
\int_{-\pi}^{\pi}
\sum_{t,s=1}^n
a_t(\lambda)\overline{a_s(\lambda)}e^{-\ii(t-s)\mu}\,\dd\mu \\
&=
\frac{1}{n}\sum_{t=1}^n |a_t(\lambda)|^2.
\end{align*}

Now let
\[
\varphi_t(\omega):=\frac{e^{-\ii t\omega}}{\sqrt{2\pi}},
\qquad t\in\mathbb Z.
\]
Then \((\varphi_t)_{t\in\mathbb Z}\) is an orthonormal basis of \(L^2([-\pi,\pi])\). By Bessel's inequality,
\[
\sum_{t=1}^n |a_t(\lambda)|^2
=
2\pi \sum_{t=1}^n |\langle \psi_\lambda,\varphi_t\rangle|^2
\le
2\pi \sum_{t\in\mathbb Z} |\langle \psi_\lambda,\varphi_t\rangle|^2
\le
2\pi \|\psi_\lambda\|_{L^2(-\pi,\pi)}^2.
\]
Substituting this bound into the previous identity yields the claim.
\end{proof}


\bibliographystyle{amsplain}
\bibliography{biblio.bib}

\end{document}